\documentclass[12pt,reqno]{amsart}
\usepackage{pifont}
\usepackage[english]{babel}
\usepackage{amsmath,amsfonts,amssymb,amsthm,amscd,latexsym}
\usepackage[T1]{fontenc}
\usepackage[utf8]{inputenc}

\usepackage{tikz}
\usetikzlibrary{arrows}
\usepackage{color}
\usepackage{cite}
\usepackage{multirow}
\usepackage{cite}

\usepackage{ifpdf}
\ifpdf
 \usepackage[colorlinks,final,backref=page,hyperindex]{hyperref}
\else
  \usepackage[colorlinks,final,backref=page,hyperindex,hypertex]{hyperref}
\fi

\newtheorem{theorem}{Theorem}[section]
\newtheorem{lemma}[theorem]{Lemma}

\newtheorem{corollary}[theorem]{Corollary}
\theoremstyle{definition}
\newtheorem{definition}[theorem]{Definition}

\theoremstyle{remark}

\numberwithin{equation}{section}

\begin{document}

\setcounter{page}{1}

\title[local derivations on Witt algebras]{local derivations on Witt algebras}

\author{Yang Chen}
\address{Mathematics Postdoctoral Research Center,
Hebei Normal University, Shijiazhuang 050016, Hebei, China}
\email{chenyang1729@hotmail.com}

\author{Kaiming Zhao}
\address{Department of Mathematics, Wilfrid
Laurier University, Waterloo, ON, Canada N2L 3C5,  and College of
Mathematics and Information Science, Hebei Normal University, Shijiazhuang 050016, Hebei, China}
\email{kzhao@wlu.ca}

\author{Yueqiang Zhao}
\address{School of
Mathematical Sciences, Hebei Normal University, Shijiazhuang 050016, Hebei, China}
\email{yueqiangzhao@163.com}

%\subjclass[2011]{Primary 46L57; 46L40} %; Secondary hsecondary classesi}

%\keywords{Lie algebras, positive Witt algebra}

\date{}
\maketitle

\begin{abstract} In this paper, we prove that every local derivation on Witt algebras $W_n, W_n^+$ or $W_n^{++} $ is a derivation for any $n\in\mathbb{N}$. As a consequence we obtain that  every local derivation on a centerless generalized Virasoro algebra of higher rank is a derivation.

\


\end{abstract}

\section{Introduction}

 Let  $A$ a Banach  (or associative) algebra, $X$ be an
$A$-bimodule. A linear mapping
$\Delta:A\to X$ is said to be a \textit{local derivation} if for
every $x$ in $A$ there exists a derivation $D_x :A\to X$,
depending on $x$, satisfying $\Delta(x) = D_x (x).$ When $X$ is taken to be $A$, such a local derivation
is called a local derivation on $A$. The   concept of
local derivation for Banach (or associative) algebras was introduced by Kadison \cite{Kad90}, Larson and Sourour \cite{LarSou} in 1990.
Since then there have been a lot of studies on local derivations on various algebras. See for example the recent papers
\cite{ AyuKud, AyuKudPer, AyuKudRak, LiuZh} and the references therein.
Local derivations on various algebras are some kind of local properties for the algebras,
which turn out to be  very interesting.
%The main problems concerning this notion are to find conditions
%under which local derivations become derivations and to present
%examples of algebras with local derivations that are not
%derivations.
 Kadison actually proved in \cite{Kad90} that each continuous local derivation
of a von Neumann algebra $M$ into a dual Banach $M$-bimodule is a
derivation, he established the existence of local derivations on the algebra
$\mathbb{C}(x)$ of rational functions  which are not derivations and
showed that any local derivation of the polynomial ring $\mathbb{C}[x_1, \cdots, x_n]$
is a derivation.
 Recently, several papers have been devoted to studying local derivations   for Lie (super)algebras.
In \cite{AyuKudRak}, Ayupov and Kudaybergenov proved that every local derivation on a finite dimensional semisimple Lie algebra over an algebraically closed field
of characteristic 0 is automatically a derivation, and gave examples of nilpotent Lie
algebra (so-call filiform Lie algebras) which admit local derivations which are not  derivations. In \cite{Yus} and \cite{AY} Ayupov and Yusupov studied 2-local derivations on univariate Witt algebras. In \cite{Zhao}, we study 2-local derivations on
multivariate Witt algebras. In the present paper we study local derivations on Witt algebras.

Let $n\in\mathbb{N}$. The Witt algebra $W_n$ of vector fields on an $n$-dimensional torus  is the derivation Lie algebra of the Laurent
polynomial algebra $A_n=\mathbb{C}[t_1^{\pm1},t_2^{\pm1},\cdots,
t_n^{\pm1}]$. Witt algebras were one of the four classes of Cartan type Lie algebras originally introduced in 1909 by Cartan \cite{C} when he studied
infinite dimensional simple Lie algebras.
Over the last two decades, the  representation theory of Witt algebras was  extensively
studied by many mathematicians and physicists; see for example \cite{BF, BMZ, GLLZ}. Very recently, Billig and Futorny
 obtained the classification for all simple Harish-Chandra
$W_n$-modules in their remarkable paper \cite{BF}.

The present paper is arranged as follows. In Section 2 we recall some known results  and establish some related properties concerning  Witt algebras. In Section 3 we prove that every local derivation on Witt algebras $W_n$ is a derivation. As a consequence we obtain that  every local derivation on a centerless generalized Virasoro algebra of higher rank is a derivation. The methods used in \cite{AyuKudRak} for finite dimensional semisimple Lie algebras no longer work for Witt algebras since Witt algebras have very different algebraic structure with finite dimensional semisimple Lie algebras. We have to establish new methods in the proofs of this section.
Finally, in Section 4 we show that the above methods and conclusions are applicable for Witt algebras $W_n^+$ and $W_n^{++}$.

Throughout this paper, we denote by $\mathbb{Z}$, $\mathbb{N}$, $\mathbb{Z}_+$ and $\mathbb{C}$ the sets of  all integers, positive integers, non-negative integers and complex numbers respectively. All algebras are over $\mathbb{C}$.

\section{The Witt algebras}

In this section we recall definitions, symbols and  some known
results for later use in this paper.

A derivation on a Lie algebra $L$ is a linear map
$D: L\rightarrow L$ which satisfies the Leibniz
law
$$
D([x,y])=[D(x),y]+[x, D(y)],\
\forall x,y\in L.$$ The set of all derivations of
$L$   is a Lie algebra and
is denoted by $\text{Der}(L)$. Clearly derivations on $L$ are local derivations, but   the converse may not be true in general. For any $a\in L$,
the map
$$\text{ad}(a): L\to L, \ \text{ad}(a)x=[a,x],\ \forall  x\in L$$ is a derivation and derivations of this form are called
\textit{inner derivations}. The set of all inner derivations of
$L$, denoted by $\text{Inn}(L),$ is a Lie  ideal of
$\text{Der}(L)$.

For $n\in\mathbb{N}$, let $A_n= \mathbb{C}[t_1^{\pm1},t_2^{\pm1},\cdots,t_n^{\pm1}]$ be the Laurent polynomial algebra and $W_n= \text{Der}(A_n)$ be the Witt algebra of vector fields  on an $n$-dimensional torus. Thus $W_n$ has a natural structure of a left $A$-module, which is free of rank $n$. Denote  $d_1= t_1\frac{\partial}{\partial{t_1}}, \ldots, d_n= t_n\frac{\partial}{\partial{t_n}}$, which form   a basis of this $A$-module:
$$W_n=\bigoplus_{i=1}^n A_n d_i.$$
Denote  $t^\alpha=t_1^{\alpha_1}\cdots t_n^{\alpha_n}$ for $\alpha= (\alpha_1, \ldots, \alpha_n)\in \mathbb{Z}^n$ and let $\{\epsilon_1, \ldots, \epsilon_n\}$ be the standard basis of $\mathbb{Z}^n$. Then we can write the Lie bracket in $W_n$ as follows:
$$[t^\alpha d_i, t^\beta d_j]= \beta_i t^{\alpha+ \beta} d_j- \alpha_j t^{\alpha+ \beta} d_i,\ i, j= 1, \ldots, n;\ \alpha, \beta\in \mathbb{Z}^n.$$
The subspace $\mathfrak{h}$ spanned by $d_1, \ldots, d_n$ is the Cartan subalgebra of $W_n$. We may write any nonzero element in $W_n$ as $\sum_{\alpha\in S} t^\alpha d_\alpha$, where $S$ is the finite subset consisting of all  $\alpha\in  \mathbb{Z}^n$ with $d_\alpha\in \mathfrak{h}\setminus\{0\}$. For $d_\alpha= c_1 d_1+ \cdots+ c_n d_n\in \mathfrak{h}$ and $\beta=(\beta_1,\beta_2,\cdots, \beta_n)\in \mathbb{Z}^n$, define $$(d_\alpha, \beta)=c_1 \beta_1+ \cdots+ c_n \beta_n.$$ Then we get the following formula
$$[t^\alpha d_\alpha, t^\beta d_\beta]= t^{\alpha+ \beta}((d_\alpha, \beta)d_\beta- (d_\beta, \alpha)d_\alpha),\,\,\forall \alpha, \beta\in\mathbb{Z}^n, d_\beta d_\alpha\in \mathfrak{h}.$$

For convenience when we write
\begin{equation}X=\sum_{\alpha\in \mathbb{Z}^n} \sum_{i=1}^nc_{\alpha, i} t^\alpha d_i\in W_n,\end{equation}
where $c_{\alpha, i}\in\mathbb{C}$, the coefficient $c_{\alpha, i}$ will be denoted by $(X)_{t^\alpha d_i}.$
We make the convention that when $X\in W_n$ is written as in (2.1) we always assume that the sum is finite, i.e., there are only finitely many $c_{\alpha, i}$ nonzero.

\begin{definition}\label{def} \textit{We call a vector $\mu\in \mathbb{C}^n$ generic if $\mu\cdot \alpha\neq 0$ for all $\alpha\in \mathbb{Z}^n \setminus \{0\}$ , where $\mu\cdot \alpha$ is the standard inner product on $\mathbb{C}^n$.}
\end{definition}

For a generic vector    $\mu=(\mu_1,\mu_2, \cdots, \mu_n) \in \mathbb{C}^n$ let $d_\mu= \mu_1 d_1+ \cdots+ \mu_n d_n$. Then we have the Lie subalgebra of $W_n$:
$$W_n(\mu)= A_n d_\mu,$$
which is called (centerless) generalized Virasoro algebra of rank $n$, see \cite{PZ}.

From Proposition 4.1 and Theorem 4.3 in \cite{DZ} we know that
any derivation on $W_n$ is inner. Then for the Witt algebra $W_n$, the above definition of the local derivation can be reformulated as follows. A linear map $\Delta$ on $W_n$ is a local derivation on $W_n$ if
for every elements $x\in W_n$ there exists an element
$a_x \in W_n$ such that $\Delta(x) = [a_x, x]$.

\section{Local derivations on $W_n$}

In this section  we shall mainly prove the following result concerning local derivations on $W_n$ for $n\in\mathbb{N}$.

\begin{theorem}\label{thm31} \textit{Every local derivation
on the Witt algebra $W_n$ is a derivation.}
\end{theorem}
Since  the proof of this theorem is long, we first setup  six lemmas as preparations. Let $\mu=(\mu_1,\mu_2, \cdots, \mu_n) $ be a fixed  generic vector in $\mathbb{C}^n$ and $d_\mu= \mu_1 d_1+ \cdots+ \mu_n d_n$.
For a given $\alpha\in \mathbb{Z}^n$, we define an equivalence relation $\stackrel{\alpha}{\sim}$ on $\mathbb{Z}^n$  {for } $\beta, \gamma\in \mathbb{Z}^n$:
$$ \beta\stackrel{\alpha}{\sim} \gamma\ \text{if and only if}\ \gamma- \beta= k\alpha,\ \text{for some }\ k\in \mathbb{Z}.$$
Let $[\gamma]:= \{\beta\in \mathbb{Z}^n\mid \gamma\stackrel{\alpha}{\sim} \beta\}$ denote the equivalence class containing $\gamma$. The set of equivalence classes of $\mathbb{Z}^n$ defined by $\alpha$ is denoted by
$\mathbb{Z}^n/ \alpha$.

Let $\Delta$ be a local derivation on $W_n$ with  $\Delta(d_\mu)= 0$. For $t^\alpha d_\mu$ with $\alpha\ne0$, since $\Delta$ is a local derivation there is  an element $a= \sum_{\beta\in \mathbb{Z}^n} t^\beta d''_\beta\in W_n$, where $d''_{\beta}\in \mathfrak{h} $, such that
\begin{equation}
\Delta(t^\alpha d_\mu)= [a, t^\alpha d_\mu]= \sum_{[\gamma]\in F}\sum_{k= p_\gamma}^{q_\gamma}t^{\gamma+ k\alpha} d_{\gamma+ k\alpha},
\end{equation}
where $F$ is a finite subset of $\mathbb{Z}^n/\alpha$ and $p_\gamma\leq q_\gamma\in \mathbb{Z}$.
It is clear that 
\begin{equation}
d_\alpha= (d''_0, \alpha)d_\mu.
\end{equation}
For $t^\alpha d_\mu+ xd_\mu$ where $x\in\mathbb{C}^*$, since  $\Delta$ is a local derivation there is  an element
$$\sum_{[\gamma]\in F}\sum_{k= p'_{\gamma}}^{q'_{\gamma}}t^{\gamma+ k\alpha} d'_{\gamma+ k\alpha}\in W_n,$$
where $d'_{\gamma+ k\alpha}\in \mathfrak{h} $ and $p'_\gamma\leq q'_\gamma\in \mathbb{Z}$, such that
\begin{equation}\aligned &\Delta(t^\alpha d_\mu)= \Delta(t^\alpha d_\mu+ xd_\mu)= [\sum_{[\gamma]\in F}\sum_{k= p'_{\gamma}}^{q'_{\gamma}}t^{\gamma+ k\alpha} d'_{\gamma+ k\alpha}, t^\alpha d_\mu+ xd_\mu]\\
= &\sum_{[\gamma]\in F}\sum_{k= p'_{\gamma}}^{q'_{\gamma}+ 1} t^{\gamma+ k\alpha}((d'_{\gamma+ (k-1)\alpha}, \alpha)d_\mu- (d_\mu, \gamma+ (k-1)\alpha)d'_{\gamma+ (k-1)\alpha}\\
&\ \ \ \ \ \ \ \ \ \ \ \ \ \ \ \ \ \ \ \ \ \ -x(d_\mu, \gamma+ k\alpha)d'_{\gamma+ k\alpha}),\endaligned\end{equation}
where we have assigned $d'_{\gamma+ (p'_{\gamma}- 1)\alpha}= d'_{\gamma+ (q'_{\gamma}+ 1)\alpha}= 0$. Note that we have the same $F$ in (3.1) and (3.3).

\begin{lemma}\label{lem32'} \textit{Let $\Delta$ be a local derivation
on $W_n$ such that $\Delta(d_\mu)= 0$.
Then $F=\{[0]\}$ in (3.1) and (3.3).}
\end{lemma}

\begin{proof} We have assumed that $\alpha\ne0$.
Suppose that $d_{\gamma+ p_\gamma \alpha}\neq 0$ and $d_{\gamma+ q_\gamma \alpha}\neq 0$, $d'_{\gamma+ p'_\gamma \alpha}\neq 0$ and $d'_{\gamma+ q'_\gamma \alpha}\neq 0$ for some $[\gamma]\neq [0]$. Comparing the right hand sides of  (3.1) and (3.3) we see  that $p_\gamma=p'_\gamma$ and $q_\gamma\le q'_\gamma+1$. If   $q_\gamma<q'_\gamma+1$,  from (3.1) and (3.3) we deduce that
$$(d'_{\gamma+ q'_\gamma \alpha}, \alpha)d_\mu- (d_\mu, \gamma+ q'_\gamma \alpha)d'_{\gamma+ q'_\gamma \alpha}= 0.$$
Since $(d_\mu, \gamma+ q'_\gamma \alpha)\neq 0$, we see  that $d'_{\gamma+ q'_\gamma \alpha}= c d_\mu$ for some $c\in \mathbb{C}^*$, and furthermore
$$(c d_\mu, \alpha)d_\mu- (d_\mu, \gamma+ q'_\gamma \alpha)c d_\mu= c(d_\mu, -\gamma- (q'_\gamma-1) \alpha)d_\mu= 0.$$
Then $ \gamma+(q'_\gamma-1) \alpha=0$, i.e., $[\gamma]= [0]$, a contradiction. Thus $q_\gamma=q'_\gamma+1$, and $p_\gamma<q_\gamma$.

Comparing (3.1) and (3.3) we deduce that (at least two equations)
\begin{equation}
\aligned &-x(d_\mu, \gamma+ p_\gamma \alpha)d'_{\gamma+ p_\gamma \alpha}= d_{\gamma+ p_\gamma \alpha};\\
&(d'_{\gamma+ p_\gamma\alpha}, \alpha)d_\mu- (d_\mu, \gamma+ p_\gamma\alpha)d'_{\gamma+ p_\gamma\alpha}- x(d_\mu, \gamma+ (p_\gamma+ 1)\alpha)d'_{\gamma+ (p_\gamma+ 1)\alpha}= d_{\gamma+ (p_\gamma+1)\alpha};\\
&  \cdots\cdots  \cdots\cdots;\\
&(d'_{\gamma+ (q_\gamma-2)\alpha}, \alpha)d_\mu- (d_\mu, \gamma+ (q_\gamma-2)\alpha)d'_{\gamma+ (q_\gamma-2)\alpha}- x(d_\mu, \gamma+ (q_\gamma-1)\alpha)d'_{\gamma+ (q_\gamma-1)\alpha}\\
&\hskip 3cm = d_{\gamma+ (q_\gamma-1)\alpha};\\
&(d'_{\gamma+ (q_\gamma-1)\alpha}, \alpha)d_\mu- (d_\mu, \gamma+ (q_\gamma-1)\alpha)d'_{\gamma+ (q_\gamma-1)\alpha}=
d_{\gamma+ q_\gamma\alpha}.\endaligned
\end{equation}
Since $(d_\mu, \gamma+ k\alpha)\neq 0$ for $k\in\mathbb{Z}$,  eliminating $d'_{\gamma+ p_\gamma\alpha}, \ldots, d'_{\gamma+ (q_\gamma- 1)\alpha}$ in this order by substitution we see that
\begin{equation}d_{\gamma+ q_\gamma\alpha}+ *x^{-1}+ \cdots+ *x^{-q_\gamma+p_\gamma}= 0,\end{equation}
where $*\in \mathfrak{h}$ are independent of $x$. We always find some $x\in \mathbb{C}^*$ not satisfying (3.5), which is a contradiction.
The lemma follows.
\end{proof}

Now (3.1) and (3.3) become
\begin{equation}
\aligned &\sum_{k= p}^{q}t^{k\alpha} d_{ k\alpha}= \sum_{k= p'}^{q'+ 1} t^{ k\alpha}((d'_{ (k-1)\alpha}, \alpha)d_\mu- (d_\mu, (k-1)\alpha)d'_{ (k-1)\alpha} -x(d_\mu,k\alpha)d'_{ k\alpha}),\endaligned\end{equation}
where $p=p_0, q=q_0, p'=p'_0, q'=q'_0$ and we have assigned $d'_{(p'- 1)\alpha}= d'_{(q'+ 1)\alpha}= 0$.
We may assume  that $d_{p \alpha}\neq 0$ and $d_{ q\alpha}\neq 0$, $d'_{ p' \alpha}\neq 0$ and $d'_{q' \alpha}\neq 0$. Clearly $p'\le p\le q\le q'+1$, and $p'=p$ if $p'\ne 0$. Our destination is to  prove that $p=q=1$.
%Now we consider two different cases.

\begin{lemma}\label{lem32''} \textit{Let $\Delta$ be a local derivation
on $W_n$ such that $\Delta(d_\mu)= 0$.
Then $p'\ge0$ and $p\ge 1$ in (3.6).}
\end{lemma}

\begin{proof}   To the contrary we assume that $p'< 0$. Then $p'= p$.
 If further $q'\geq -1$, from (3.6) we obtain a set of (at least two) equations
\begin{equation}
\aligned &-x(d_\mu, p \alpha)d'_{p \alpha}= d_{p \alpha};\\
&(d'_{p\alpha}, \alpha)d_\mu- (d_\mu, p\alpha)d'_{p\alpha}- x(d_\mu, (p+ 1)\alpha)d'_{(p+ 1)\alpha}= d_{(p+1)\alpha};\\
&\cdots\cdots;\\
&(d'_{-\alpha}, \alpha)d_\mu- (d_\mu, -\alpha)d'_{-\alpha}= d_0.\endaligned
\end{equation}
If $d_0\ne 0$, using the same arguments as for (3.4), the equations (3.7) makes contradictions.
So we consider the case that $d_0=0$. From the last equation in (3.7)  we see that $  d'_{-\alpha}=0$. We continue upwards in (3.7)  in this manner to some step. We get
$$d_0= d_{-\alpha}= \cdots= d_{l\alpha}= 0,\,\,\,
 d_{(l-1)\alpha}\neq 0, \ \text{and}\ d'_{-\alpha}= \cdots= d'_{(l-1)\alpha}= 0.$$
If $p+1< l (\leq 0),$
then  (3.7) becomes
\begin{equation}
\aligned &-x(d_\mu, p \alpha)d'_{p \alpha}= d_{p \alpha};\\
&(d'_{p\alpha}, \alpha)d_\mu- (d_\mu, p\alpha)d'_{p\alpha}- x(d_\mu, (p+ 1)\alpha)d'_{(p+ 1)\alpha}= d_{(p+1)\alpha};\\
&\cdots\cdots;\\
&(d'_{(l-2)\alpha}, \alpha)d_\mu- (d_\mu, (l-2)\alpha)d'_{(l-2)\alpha}= d_{(l-1)\alpha}.\endaligned
\end{equation}
Using the same arguments as for (3.4), the equations (3.8) makes contradictions. We need only to consider the case that $p+1=l$, i.e., $l-1=p$.
In this case we have that  $0=d'_{(l-1)\alpha}=d'_{p'\alpha}\ne0$, again a  contradiction.
Therefore $q'< -1$.

 If $q<q'+1$, from (3.6) we see that
$$(d'_{q' \alpha}, \alpha)d_\mu- (d_\mu, q' \alpha)d'_{q' \alpha}= 0.$$
Since $(d_\mu,  q' \alpha)\neq 0$, we see  that $d'_{  q'  \alpha}= c d_\mu$ for some $c\in \mathbb{C}^*$, and furthermore
$$(c d_\mu, \alpha)d_\mu- (d_\mu,   q'  \alpha)c d_\mu= c(d_\mu, - (q' -1) \alpha)d_\mu= 0.$$
Then $  (q'-1) \alpha=0$, i.e., $q'=1$, a contradiction. So
  $q=q'+1$ and $p<q$. We  obtain a set of  (at least two) equations from (3.6)
\begin{equation}
\aligned &-x(d_\mu, p \alpha)d'_{p \alpha}= d_{p \alpha};\\
&(d'_{p\alpha}, \alpha)d_\mu- (d_\mu, p\alpha)d'_{p\alpha}- x(d_\mu, (p+ 1)\alpha)d'_{(p+ 1)\alpha}= d_{(p+1)\alpha};\\
&\cdots\cdots;\\
&(d'_{(q-1)\alpha}, \alpha)d_\mu- (d_\mu, (q-1)\alpha)d'_{(q-1)\alpha}= d_{q\alpha}.\endaligned
\end{equation}
Using the same arguments as for (3.4), the equations (3.9) makes contradictions. Hence
$p'\geq 0.$

If $p'\ge1$, then $p= p'\ge 1$.
If $p'= 0$, then $d_0= (d'_{-\alpha}, \alpha)d_\mu- (d_\mu, -\alpha)d'_{-\alpha}= 0$ by (3.6). So
$p\geq 1$ also. \end{proof}

\begin{lemma}\label{lem32} \textit{Let $\Delta$ be a local derivation
on $W_n$ such that $\Delta(d_\mu)= 0$.
Then $$\Delta(t^\alpha d_\mu)\in\mathbb{C}t^\alpha d_\mu,\,\,\,\forall \alpha\in \mathbb{Z}^n.$$}
\end{lemma}

\begin{proof}  From Lemma \ref{lem32''}, we know that $q\ge p\ge 1$. We need only to prove that $q=1$ in (3.6). Otherwise we assume that $q> 1$, and then $q'>0$.

 {\bf Case 1:} $q'>1$.

 In this case we can show that $q=q'+ 1$ as in  the above arguments.
If $p'\geq 1$, we see that $p=p'$ and $p<q$. From (3.6) we obtain a set of  (at least two) equations
\begin{equation}
\aligned &-x(d_\mu, p \alpha)d'_{p \alpha}= d_{p \alpha};\\
&(d'_{p\alpha}, \alpha)d_\mu- (d_\mu, p\alpha)d'_{p\alpha}- x(d_\mu, (p+ 1)\alpha)d'_{(p+ 1)\alpha}= d_{(p+1)\alpha};\\
&\cdots\cdots;\\
&(d'_{(q-1)\alpha}, \alpha)d_\mu- (d_\mu, (q-1)\alpha)d'_{(q-1)\alpha}= d_{q\alpha}.\endaligned
\end{equation}
Using the same arguments as for (3.4), the equation (3.10) makes contradictions.
 So $p'=0.$ Now we have $$p'=0, p\ge 1, q=q'+1>2.$$ By (3.6) and (3.2) we have
\begin{equation}
(d'_0, \alpha)d_\mu- x(d_\mu, \alpha)d'_\alpha= d_\alpha= (d''_0, \alpha)d_\mu;
\end{equation}
\begin{equation}
(d'_\alpha, \alpha)d_\mu- (d_\mu, \alpha)d'_\alpha- x(d_\mu, 2\alpha)d'_{2\alpha}= d_{2\alpha}.
\end{equation}
Equation (3.11) implies that $d'_\alpha= c d_\mu$, $c\in \mathbb{C}$. Then Equation (3.12) can be simplified to $-x(d_\mu,2\alpha)d'_{2\alpha}= d_{2\alpha}$. Again from (3.6) we obtain a set of   (at least two)   equations
\begin{equation}
\aligned &-x(d_\mu,2\alpha)d'_{2\alpha}= d_{2\alpha}\\
&(d'_{2\alpha}, \alpha)d_\mu- (d_\mu, 2\alpha)d'_{2\alpha}- x(d_\mu, 3\alpha)d'_{3\alpha}= d_{3\alpha}\\
&\cdots\cdots\\
&(d'_{(q-1)\alpha}, \alpha)d_\mu- (d_\mu, (q-1)\alpha)d'_{(q-1)\alpha}= d_{q\alpha}.\endaligned
\end{equation}
Using the same arguments again,  (3.13) makes contradictions. So $q'=1$. Now we have

{\bf Case 2:} $q'=1$.

We need only consider the case that $q=2$. In this case we still have Equations (3.11) and (3.12) with $d'_{2\alpha}=0$. Equation (3.11) implies that $d'_\alpha= c d_\mu$, $c\in \mathbb{C}$. Then Equation (3.12)  implies that $d_{2\alpha}=0$ which is a contradiction.

 Therefore  $p= q= 1$ and $\Delta(t^\alpha d_\mu)= t^\alpha d_\alpha= (d''_0, \alpha)t^\alpha d_\mu$ by (3.2).
The lemma follows.
\end{proof}

\begin{lemma}\label{lem33}
 \textit{Let $\Delta$ be a local derivation on $W_n$ such that $\Delta(d_\mu)= \Delta(t_i d_\mu)= 0$ for a given $1\leq i\leq n$.
 Then $\Delta(t_i^m d_\mu)= 0$ for any $m\in \mathbb{Z}$.}
\end{lemma}

\begin{proof}
We may assume that $m\neq 0, 1$. By Lemmas \ref{lem32}, there is $c\in \mathbb{C}$ and
$$\sum_{[\gamma]\in F}\sum_{k= p_{\gamma}}^{q_{\gamma}}t^{\gamma+ k(m-1)\epsilon_i} d_{\gamma+ k(m-1)\epsilon_i}\in W_n,$$
where $F$ is a finite subset of $\mathbb{Z}^n/(m-1)\epsilon_i$ and $p_\gamma\leq q_\gamma\in \mathbb{Z}$, $d_{\gamma+ k(m-1)\epsilon_i}\in \mathfrak{h}$, such that
\begin{equation}
\aligned ct_i^m d_\mu=& \Delta(t_i^m d_\mu)= \Delta(t_i^m d_\mu+ t_i d_\mu)=
[\sum_{[\gamma]\in F}\sum_{k= p_{\gamma}}^{q_{\gamma}}t^{\gamma+ k(m-1)\epsilon_i} d_{\gamma+ k(m-1)\epsilon_i}, t_i^m d_\mu+ t_i d_\mu]\\
=& [\sum_{k= p_0}^{q_0} t_i^{k(m-1)} d_{k(m-1)\epsilon_i}+ \sum_{[0]\neq [\gamma]\in F}\sum_{k= p_{\gamma}}^{q_{\gamma}}t^{\gamma+ k(m-1)\epsilon_i} d_{\gamma+ k(m-1)\epsilon_i}, t_i^m d_\mu+ t_i d_\mu]\\
=& [\sum_{k= p_0}^{q_0} t_i^{k(m-1)} d_{k(m-1)\epsilon_i}, t_i^m d_\mu+ t_i d_\mu]\\
=& \sum_{k= p_0}^{q_0+1}t_i^{k(m-1)+1}((d_{(k-1)(m-1)\epsilon_i}, m\epsilon_i)d_\mu- (d_\mu, (k-1)(m-1)\epsilon_i)d_{(k-1)(m-1)\epsilon_i})\\
&\ \ \ \ \ \ \ \ \ \ \ \ \ \ \ \ + ((d_{k(m-1)\epsilon_i}, \epsilon_i)d_\mu- (d_\mu, k(m-1)\epsilon_i)d_{k(m-1)\epsilon_i}).\endaligned
\end{equation}
where we have assigned $d_{(p_0-1)(m-1)\epsilon_i}= d_{(q_0+1)(m-1)\epsilon_i}= 0$, and we have used the fact  that
$$[\sum_{[0]\neq [\gamma]\in F}\sum_{k= p_{\gamma}}^{q_{\gamma}}t^{\gamma+ k(m-1)\epsilon_i} d_{\gamma+ k(m-1)\epsilon_i}, t_i^m d_\mu+ t_i d_\mu]= 0$$
since it cannot contain elements from $t_i^m \mathfrak{h}$ or eliminate any term in
$$[\sum_{k= p_0}^{q_0} t_i^{k(m-1)} d_{k(m-1)\epsilon_i}, t_i^m d_\mu+ t_i d_\mu].$$
We may assume that $d_{p_0(m-1)\epsilon_i}\neq 0$
and $d_{q_0(m-1)\epsilon_i}\neq 0$.
Our destination is to prove that $c=0$, To the contrary we assume that $c\ne0$. Then
$$p_0\le 1,   q_0\ge0,\text{ and }p_0\le q_0.$$

{\bf Claim 1.}  $p_0= 0$ or $1$.

Suppose that $p_0<0$, by (3.14) we deduce that
$$(d_{p_0(m-1)\epsilon_i}, \epsilon_i)d_\mu- (d_\mu, p_0(m-1)\epsilon_i)d_{p_0(m-1)\epsilon_i}= 0.$$
Since $(d_\mu, p_0(m-1)\epsilon_i)\neq 0$, we have $d_{p_0(m-1)\epsilon_i}
= c'd_\mu$ for some $c'\in \mathbb{C}^*$ and furthermore
$$\aligned &(d_{p_0(m-1)\epsilon_i}, \epsilon_i)d_\mu- (d_\mu, p_0(m-1)\epsilon_i)d_{p_0(m-1)\epsilon_i}\\
= &(c'd_\mu, \epsilon_i)d_\mu- (d_\mu, p_0(m-1)\epsilon_i)c'd_\mu\\
= &c'(d_\mu, (1- p_0(m-1))\epsilon_i)d_\mu= 0\endaligned.$$
It follows that $1- p_0(m-1)= 0$, and thus $p_0=-1, m=0$, a contradiction. Hence $p_0= 0$ or $1$.

{\bf Claim 2.}  $q_0= 0$.

Suppose that $q_0> 0$, by (3.14) we deduce that
$$(d_{q_0(m-1)\epsilon_i}, m\epsilon_i)d_\mu- (d_\mu, q_0(m-1)\epsilon_i)d_{q_0(m-1)\epsilon_i}= 0.$$
Since $(d_\mu, q_0(m-1)\epsilon_i)\neq 0$, we have $d_{q_0(m-1)\epsilon_i}
= c'd_\mu$ for some $c'\in \mathbb{C}^*$ and furthermore
$$\aligned &(d_{q(m-1)\epsilon_i}, m\epsilon_i)d_\mu- (d_\mu, q(m-1)\epsilon_i)d_{q(m-1)\epsilon_i}\\
= &(c'd_\mu, m\epsilon_i)d_\mu- (d_\mu, q_0(m-1)\epsilon_i)c'd_\mu\\
= &c'(d_\mu, (m- q_0(m-1))\epsilon_i)d_\mu= 0\endaligned.$$
It follows that $m- q_0(m-1)= 0$, and thus $m= 2$ and $q_0= 2$. Now from (3.14) we deduce that
\begin{align}
&(d_0, \epsilon_i)d_\mu= 0;\\
&(d_0, 2\epsilon_i)d_\mu+ ((d_{\epsilon_i}, \epsilon_i)d_\mu- (d_\mu, \epsilon_i)d_{\epsilon_i})= cd_\mu;\\
&((d_{\epsilon_i}, 2\epsilon_i)d_\mu- (d_\mu, \epsilon_i)d_{\epsilon_i})+ ((d_{2\epsilon_i}, \epsilon_i)d_\mu- (d_\mu, 2\epsilon_i)d_{2\epsilon_i})= 0;\\
&(d_{2\epsilon_i}, 2\epsilon_i)d_\mu- (d_\mu, 2\epsilon_i)d_{2\epsilon_i}= 0.
\end{align}
We see that $d_{\epsilon_i}= c''d_\mu$, $c''\in \mathbb{C}$ by (3.18) and (3.17). Substituting into (3.16)
it follows that $c= 0$, a contradiction. Claim 2 follows.

From Claims 1 and 2 we have   $p_0= q_0= 0$. By (3.14) we have $(d_0, \epsilon_i)d_\mu= 0$ and still $cd_\mu= (d_0, m\epsilon_i)d_\mu= 0$. The proof is completed.
\end{proof}

\begin{lemma}\label{lem34}
 \textit{Let $\Delta$ be a local derivation on $W_n$ such that $\Delta(d_\mu)= \Delta(t_i d_\mu)= 0$ for all $1\leq i\leq n$. Then $\Delta(t^\alpha d_\mu)= 0$ for all $\alpha\in\mathbb{Z}^n$.}
\end{lemma}

\begin{proof}
 From Lemma \ref{lem33}, we have
$$\Delta(t^{m\epsilon_i} d_\mu)= 0, \forall 1\le i\le n,  m\in \mathbb{Z}.$$
We may assume that $n>1$ and $\alpha\in \mathbb{Z}^n\setminus \cup_{1\le i\le n} \mathbb{Z}\epsilon_i$.
Take an fixed integer $m> |\alpha_i|$ for any $1\leq i\leq n$. Define
$$I=\{i:\alpha_i\geq 0\}\ \text{and}\ I'=\{i: \alpha_i< 0\}.$$
Then by Lemma \ref{lem32} and Lemma \ref{lem33} there is  $c\in \mathbb{C}$ and an element $ x=\sum_{\beta\in \mathbb{Z}^n} t^\beta d_\beta\in W_n$ such that
\begin{equation}\aligned ct^\alpha d_\mu&= \Delta(t^\alpha d_\mu)= \Delta(t^\alpha d_\mu+ \sum_{i\in I}t_i^m d_\mu+ \sum_{i\in I'}t_i^{-m} d_\mu)\\
&= [\sum_{\beta\in \mathbb{Z}^n} t^\beta d_\beta, t^\alpha d_\mu+ \sum_{i\in I}t_i^m d_\mu+ \sum_{i\in I'}t_i^{-m} d_\mu].\endaligned\end{equation}

{\bf Claim 1.} If $[t^\gamma d_\gamma, t_i^k d_\mu]= 0$ for some $1\leq i\leq n$, where $\gamma\in \mathbb{Z}^n\setminus\{0\}$, $d_\gamma\in \mathfrak{h}\setminus\{0\}$ and $k\in \mathbb{Z}$,
then $t^\gamma d_\gamma= c_it_i^k d_\mu$ for some $c_i\in\mathbb{C}^*$.

Since
$$[t^\gamma d_\gamma, t_i^k d_\mu]= t^{\gamma+ k\epsilon_i}((d_\gamma, k\epsilon_i)d_\mu- (d_\mu, \gamma)d_\gamma)= 0.$$
and $(d_\mu, \gamma)\neq 0$, we have $d_\gamma= c_i d_\mu$ for some $c_i\in \mathbb{C}^*$ and furthermore
$$(d_\gamma, k\epsilon_i)d_\mu- (d_\mu, \gamma)d_\gamma= (c_i d_\mu, k\epsilon_i)d_\mu- (d_\mu, \gamma)c_i d_\mu= c_i(d_\mu, k\epsilon_i- \gamma)d_\mu= 0.$$
It implies that $\gamma= k\epsilon_i$. Claim 1 follows.

Take a nonzero term $t^\gamma d_\gamma$ in the expression of $x$ with maximal degree with respect to  $t_i$ for some $i\in I$.

{\bf Claim 2.} If $\gamma_i\ge 0$ and $\gamma\ne0$, then $t^\gamma d_\gamma=c_it_i^m d_\mu$ for some $c_i\in\mathbb{C}^*$.

The term $[t^\gamma d_\gamma, t_i^m d_\mu]$ is of maximal degree with respect to $t_i$ in the expression of
$$[\sum_{\beta\in \mathbb{Z}^n} t^\beta d_\beta, t^\alpha d_\mu+ \sum_{i\in I}t_i^m d_\mu+ \sum_{i\in I'}t_i^{-m} d_\mu].$$
Since $\gamma+m\epsilon_i\ne\alpha$, it follows that $[t^\gamma d_\gamma, t_i^m d_\mu]= 0$. We see that $t^\gamma d_\gamma= c_it_i^m d_\mu$ for some $c_i\in\mathbb{C}^*$ by Claim 1.

Similarly,  a nonzero term $t^\gamma d_\gamma$ in the expression of $x$ with minimal degree with respect to  $t_i$  for  $i\in I'$  is $c_it_i^{-m} d_\mu$ for some  $c_i\in\mathbb{C}^*$ if  $\gamma_i\le 0$ and $\gamma\ne0$.
Moreover, $c_it_i^m d_\mu$ (resp. $c_it_i^{-m} d_\mu$) is the only possible term of   non-negative  maximal degree (resp. non-positive minimal degree) with respect to $t_i$, $i\in I$ (resp. $i\in I'$) in $x$. If there is such a term, we consider $c_it_i^m d_\mu$, $i\in I$ without loss of generality. Now, to delete the term $[c_it_i^m d_\mu, t^\alpha d_\mu]= c_i(d_\mu, \alpha- m\epsilon_i)t^{\alpha+ m\epsilon_i} d_\mu\neq 0$ in
$$[\sum_{\beta\in \mathbb{Z}^n} t^\beta d_\beta, t^\alpha d_\mu+ \sum_{i\in I}t_i^m d_\mu+ \sum_{i\in I'}t_i^{-m} d_\mu],$$
there must be $c_it^\alpha d_\mu$ in $\sum_{\beta\in \mathbb{Z}^n} t^\beta d_\beta$. To delete the term $[c_it^\alpha d_\mu, t_j^m d_\mu]\neq 0$, $i\neq j\in I$ (resp. $[c_it^\alpha d_\mu, t_j^{-m} d_\mu]$, $j\in I'$) in
$$[\sum_{\beta\in \mathbb{Z}^n} t^\beta d_\beta, t^\alpha d_\mu+ \sum_{i\in I}t_i^m d_\mu+ \sum_{i\in I'}t_i^{-m} d_\mu],$$
there must be $c_it_j^m d_\mu$, $i\neq j\in I$ (resp. $c_it_j^{-m} d_\mu$, $j\in I'$).
Thus $c_i=c_1$ for all $1\le i\le n$.

Let $x'= x- c_1(t^\alpha d_\mu+ \sum_{i\in I}t_i^m d_\mu+ \sum_{i\in I'}t_i^{-m} d_\mu)$. Then
\begin{equation}
ct^\alpha d_\mu= [x', t^\alpha d_\mu+ \sum_{i\in I}t_i^m d_\mu+ \sum_{i\in I'}t_i^{-m} d_\mu].
\end{equation}
Note that, if $t^\gamma d_\gamma$ is a nonzero  term in the expression of $x'$, then $\gamma=0$ or \begin{equation}
\left\{\begin{matrix}
\gamma_i <0 &\text{ if }i\in I, \\
\gamma_i >0& \text{ if }i\in I'.
\end{matrix}\right.
\end{equation}

{\bf Claim 3.}  In the expression of $x'$, the term $t^0d_0=0$.

If $d_0\ne0$,
considering the highest (resp. lowest) degree term with respect to $t_i$ for $i\in I$ (resp. $i\in I'$) in (3.20), we deduce that $(d_0, \epsilon_i)=0$  for all $1\le i\le n$. Thus $d_0=0$. Claim 3 follows.

Suppose that there exists a nonzero term $t^\gamma d_\gamma$ in the expression of $x'$ with maximal degree with respect to some $t_i$ for $i\in I$. Then $\gamma\neq 0$, and the term $[t^\gamma d_\gamma, t_i^m d_\mu]$ is of maximal degree with respect to $t_i$ in the expression of
$$[x', t^\alpha d_\mu+ \sum_{i\in I}t_i^m d_\mu+ \sum_{i\in I'}t_i^{-m} d_\mu].$$
It is only possible that $0\neq [t^\gamma d_\gamma, t_i^m d_\mu]\in t^\alpha\mathfrak{h}$ by Claim 1.
Then $\gamma+ m\epsilon_i= \alpha$. We have $0> \gamma_j= \alpha_j\geq 0$, $i\neq j\in I$ and $0< \gamma_j= \alpha_j< 0$, $i\neq j\in I'$ by (3.21), a contradiction. So $x'= 0$. Therefore $\Delta(t^\alpha d_\mu)= 0$.\end{proof}

\begin{lemma}\label{lem35} \textit{Let $\Delta$ be a local derivation
on $W_n$ such that $\Delta(d_\mu)=0$.
Then $\Delta\mid_\mathfrak{h}= 0$}.
\end{lemma}

\begin{proof} This is trivial for $n=1$. Next we assume that $n>1$.
For a given $i\in \{1,\dots,n\}$, there is  an element $\sum_{\alpha\in \mathbb{Z}^n} t^\alpha d_\alpha^{(i)}\in W_n$, where $d_\alpha^{(i)}\in \mathfrak{h}$ such that
$$\Delta(d_i)= [\sum_{\alpha\in \mathbb{Z}^n} t^\alpha d_\alpha^{(i)}, d_i]= -\sum_{\alpha\in \mathbb{Z}^n} \alpha_i t^\alpha d_\alpha^{(i)}.$$
So we may suppose that
$$\Delta(d_i)= \sum_{\alpha\in \mathbb{Z}^n\setminus\{0\}}\sum_{1\leq j\leq n} c^{(i)}_{\alpha,j}t^\alpha d_j,$$
where $ c^{(i)}_{\alpha,j}\in\mathbb{C}$.
For a given $\alpha\in \mathbb{Z}^n\setminus\{0\}$, we have $\alpha_k\neq 0$ for some $1\leq k \leq n$. Let $c_{\alpha, j}= c^{(k)}_{\alpha,j}/\alpha_k$.
For $i\neq k$, let $d= \alpha_k d_i- \alpha_i d_k$. We have
$$(\Delta(d))_{t^\alpha d_j}= \alpha_k c^{(i)}_{\alpha,j}- \alpha_i c^{(k)}_{\alpha,j}.$$
On the other hand, there is an element $\sum_{\alpha\in \mathbb{Z}^n}\sum_{1\leq j\leq n} b_{\alpha,j}t^\alpha d_j\in W_n$ where $b_{\alpha,j}\in\mathbb{C}$
such that
$$(\Delta(d))_{t^\alpha d_j}= ([\sum_{\alpha\in \mathbb{Z}^n}\sum_{1\leq j\leq n} b_{\alpha,j}t^\alpha d_j, \alpha_k d_i- \alpha_i d_k])_{t^\alpha d_j}
= -b_{\alpha,j}\alpha_k\alpha_i+ b_{\alpha,j}\alpha_i\alpha_k= 0.$$
Then $c^{(i)}_{\alpha,j}= \alpha_i(c^{(k)}_{\alpha,j}/\alpha_k)= \alpha_i c_{\alpha,j}$, yielding that
$$0=( \Delta(d_\mu))_{t^\alpha d_j}= \sum_{i=1}^n\mu_i c^{(i)}_{\alpha,j}= c_{\alpha,j}\sum_{i=1}^n\mu_i\alpha_i.$$
We deduce that $c_{\alpha,j}= 0$. So $$c^{(i)}_{\alpha,j}= 0,\,\, \forall   1\leq i,j\leq n, \alpha\in\mathbb{Z}^n\setminus\{0\},$$   that is,  $\Delta(d_i)=0$ for $1\leq i\leq n$. Therefore $\Delta\mid_\mathfrak{h}= 0$.
\end{proof}

Now we are in the position to prove Theorem \ref{thm31}.

\

\textit{Proof of Theorem} \ref{thm31}. Let $\Delta$ be a local derivation on $W_n$. We fix an arbitrary  generic $\mu\in\mathbb{C}^n$. There is an element $a\in W_n$ such that $\Delta(d_\mu)=[a, d_\mu]$.
Set $\Delta_1= \Delta- \text{ad}(a)$. Then $\Delta_1$ is a local derivation such that $\Delta_1(d_\mu)= 0$. From Lemma \ref{lem35}, we know that $\Delta_1(\mathfrak{h})=0$.
By Lemma \ref{lem32}, there are $c_i\in \mathbb{C}$ such  that
$$\Delta_1(t_i d_\mu)= c_i t_i d_\mu,\,\,\forall 1\leq i\leq n.$$
Set $\Delta_2= \Delta_1- \sum_{i=1}^n c_i \text{ad}(d_i)$. Then $\Delta_2$ is a local derivation such that $$\Delta_2(\mathfrak{h})=0, \text{ and }\Delta_2(t_i d_\mu)= 0,\,\,\forall 1\leq i\leq n.$$
By Lemma \ref{lem34}, for any generic $\mu\in\mathbb{C}^n$
we have $$\Delta_2(t^\alpha d_\mu)= 0,\,\,\forall \alpha\in\mathbb{Z}^n.$$
For any generic $\lambda\in\mathbb{C}^n$ that is not a multiple of $\mu$ since $\Delta_2(d_{\lambda})=0$, from Lemma \ref{lem32}
there is $c_\alpha\in \mathbb{C}$ such that  $\Delta_2(t^\alpha d_{\lambda})= c_\alpha t^\alpha d_{\lambda}$.
There is $\sum_{\beta\in \mathbb{Z}^n}t^\beta d_\beta$, where $d_\beta\in\mathfrak{h}$,  such that
$$\aligned c_\alpha t^\alpha d_{\lambda}=& \Delta_2(t^\alpha d_\alpha)= \Delta_2(t^\alpha d_\mu+ t^\alpha d_{\lambda})
=[\sum_{\beta\in \mathbb{Z}^n}t^\beta d_\beta, t^\alpha d_\mu+ t^\alpha d_{\lambda}]\\
=&\sum_{\beta\in \mathbb{Z}^n} t^{\alpha+\beta}((d_\beta, \alpha)(d_\mu+ d_{\lambda})-(d_\mu+ d_{\lambda},\beta)d_\beta).\endaligned$$
We see that $c_\alpha t^\alpha d_{\lambda}=(d_0, \alpha)(d_\mu+ d_{\lambda})$, yielding that
  $c_\alpha= 0$. Thus  for any generic
 vector $\lambda$,
 $$\Delta_2(t^\alpha d_{\lambda})= 0,\,\,\forall \alpha\in\mathbb{Z}^n.$$
Since the set $\{d_{\lambda}:\lambda\ \text{is generic}\}$ can span $\mathfrak{h}$ we must have  $\Delta_2= 0$.
Hence $\Delta= \text{ad}(a)+ \sum_{i=1}^n c_i \text{ad}(d_i)$ is a derivation. The proof is completed.
\hfill$\Box$

By Theorem 3.4 in \cite{DZ1} any derivation on the generalized Virasoro algebra $W_n(\mu)$ can be seen as the restriction
of a inner derivation on $W_n$.  All the proofs in this section with minor modifications are valid
for the generalized Virasoro algebra $W_n(\mu)$. Therefore
we obtain the following consequence.

\begin{corollary}\label{lem36} \textit{
Let  $n\in\mathbb{N}$, and let $\mu\in\mathbb{C}^n$ be generic. Then any local derivation
on the generalized Virasoro algebra $W_n(\mu)$  is a derivation.
}
\end{corollary}

\section{Local derivations on $W_n^+$ and $W_n^{++}$}

For  $n\in\mathbb{N}$, we have the Witt algebra $W_n^+=\text{Der}(\mathbb{C}[t_1 ,t_2 ,\cdots, t_n])$ which is a subalgebra of $W_n$. We use $ \mathfrak{h}$ to denote the Cartan subalgebra of $W_n$ which is also a  Cartan subalgebra (not unique) of $W_n^+$. We know that
$$W_n^+= \sum_{\alpha\in \mathbb{Z}_+^n} t^\alpha \mathfrak{h}+ \sum_{i=1}^n \mathbb{C} t_i^{-1} d_i.$$
Furthermore $W_n^+$ has a subalgebra
$$W_n^{++}= \sum_{\alpha\in \mathbb{Z}_+^n} t^\alpha \mathfrak{h}.$$
It is well-known that $W_n^+$ is a simple Lie algebra, but $W_n^{++}$ is not.

From Proposition 4.1 and Theorem 4.3 in \cite{DZ} we know that any derivation on $W_n^+$ is inner .
Using same arguments as the proof of  Proposition 3.3 in \cite{DZ1} we can show that
any derivation on  $W_n^{++}$ is inner.

Hence, the proofs and conclusions with slight modifications  in Section 3 are applicable to $W_n^+$ and $W_n^{++}$. It is routine to verify this. We omit the details and directly    state   the following theorem.

\begin{theorem}\label{thm51} \textit{Every local derivation
on Witt algebras $W_n^+$ or $\ W_n^{++}$ is a derivation.}
\end{theorem}

\vspace{2mm}
\noindent
{\bf Acknowledgements. } This research is partially supported by NSFC (11871190) and NSERC (311907-2015).


\begin{thebibliography}{22}

\bibitem{AyuKud}  Sh. A. Ayupov,  K.  K.  Kudaybergenov,  Local derivations on measurable operators and commutativity, Eur. J. Math. 2 (2016), no. 4, 1023-1030.

\bibitem{AyuKudPer}   Sh. A. Ayupov,   K. Kudaybergenov,  A. Peralta,   A survey on local and 2-local derivations on $\mathbb{C}^*$- and von Neumann algebras. Topics in functional analysis and algebra, 73-126, Contemp. Math., 672, Amer. Math. Soc., Providence, RI, 2016.


\bibitem{AyuKudRak} Sh. A. Ayupov,  K.  K.  Kudaybergenov,  Local derivations on finite-dimensional Lie  algebras, Linear Algebra Appl. 493 (2016), 381-398.

\bibitem{AY} Sh. A. Ayupov,  B.  B.  Yusupov,  2-Local derivations on infinite-dimensional Lie algebras, arXiv:1901.04261.

\bibitem{BF} Y. Billig, V. Futorny, Classification of simple $W_n$-modules with
finite-dimensional weight spaces,  J. Reine Angew. Math. 720 (2016), 199-216.

\bibitem{BMZ} Y. Billig, A. Molev, R. Zhang, Differential equations in vertex algebras and
simple modules for the Lie algebra of vector fields on a torus, Adv.
Math., 218 (2008), no.6, 1972-2004.

\bibitem{C} E. Cartan,   Les groupes de transformations continus, infinis, simples. (French) Ann. Sci. École Norm. Sup. (3) 26 (1909), 93-161.

\bibitem{DZ}D. Z. Djokovic, K. Zhao, Generalized Cartan type $W$ Lie
algebras in characteristic zero, J. Algebra, 195 (1997), 170-210.

\bibitem{DZ1}D. Z. Djokovic, K. Zhao, Derivations, isomorphisms and second cohomology of generalized Witt algebra, Trans. Amer. Math. Soc., (2) 350 (1998), 643-664.

\bibitem{GLLZ} X. Guo, G. Liu, R. Lu, K. Zhao,  Simple Witt modules that are finitely generated over the Cartan subalgebra,  to appear in  Moscow Mathematical Journal.

\bibitem{Kad90} R.V. Kadison, Local derivations, J. Algebra, 130 (1990), 494-509.

\bibitem{LarSou} D.R. Larson and A.R. Sourour, Local derivations and local automorphisms of $B(X)$, Proc. Sympos. Pure Math., 51, Part 2, Providence, Rhode Island 1990, 187-194.

\bibitem{LiuZh} D. Liu,   J. Zhang,  Local Lie derivations of factor von Neumann algebras, Linear Algebra Appl. 519 (2017), 208-218.

\bibitem{PZ} J. Patera, H. Zassenhaus, The higher rank Virasoro algebras, Comm. Math. Phys. 136 (1991), 1-14.

\bibitem{Yus} B. B. Yusupov,  2-local derivations on Witt algebras, Uzbek Math. J., 125 (2018), 160-166.

\bibitem{Zhao} Y. Zhao, Y. Chen, K. Zhao, 2-Local derivations on Witt algebras, arXiv:1909.06242.
\end{thebibliography}
\end{document}